\documentclass[12pt]{elsarticle}
\usepackage{amsmath,amsfonts,amssymb,latexsym,graphics,epsfig,url}
\usepackage{color}
\usepackage{amsthm}
\usepackage{multirow}
\usepackage{longtable}

\usepackage[english]{babel}
\usepackage{epsfig}
\usepackage{graphicx}
\usepackage{bm}
\newcommand{\executeiffilenewer}[3]{%
\ifnum\pdfstrcmp{\pdffilemoddate{#1}}%
{\pdffilemoddate{#2}}>0%
{\immediate\write18{#3}}\fi%
}

\usepackage{fullpage}

\journal{Linear Algebra and its Applications}

\newtheorem{theo}{Theorem}[section]

\newtheorem{lema}[theo]{Lemma}

\newtheorem{coro}[theo]{Corollary}

\def\matrix0{{\mbox {\boldmath $O$}}}

\def\vec0{\mbox{\bf 0}}

\begin{document}

\begin{frontmatter}
\title{Spectral bounds for the $k$-independence number of a graph}

\author{Aida Abiad\corref{aida}}\ead{A.Abiad.Monge@maastrichtuniversity.nl}

\author{Sebastian M.\ Cioab\u{a}\corref{sebi}}\ead{cioaba@udel.edu}

\author{Michael Tait\corref{mike}}\ead{mtait@cmu.edu}

\cortext[aida]{Dept. of Quantitative Economics, Operations Research, Maastricht University, Maastricht, The Netherlands.}
\cortext[sebi]{Dept. of Math. Sciences, University of Delaware, Newark, DE 19707, USA. Research partially supported by NSF grant DMS-1600768}
\cortext[mike]{Dept. of Math. Sciences, Carnegie Mellon University, Pittsburgh, PA 15224, USA. Research partially supported by NSF grant DMS-1606350}

\date{}

\begin{abstract}
In this paper, we obtain two spectral upper bounds for the $k$-independence number of a graph which is is the maximum size of a set of vertices at pairwise distance greater than $k$. We construct graphs that attain equality for our first bound and show that our second bound compares favorably to previous bounds on the $k$-independence number.

\end{abstract}

 \end{frontmatter}
 

\section{Introduction}

The independence number of a graph $G$, denoted by $\alpha (G)$, is the size of the largest independent set of vertices in $G$. A natural generalization of the independence number is the $k$-\emph{independence number} of $G$, denoted by $\alpha_k (G)$ with $k\ge 0$, which is the maximum number of vertices that are mutually at distance greater than $k$. Note that $\alpha_0(G)$ equals the number of vertices of $G$ and $\alpha_1(G)$ is the independence number of $G$.

The $k$-independence number is an interesting graph theoretic parameter that is closely related to coding theory, where codes and anticodes are $k$-independent sets in Hamming graphs (see \cite[Chapter 17]{MWS}). In addition, the $k$-independence number of a graph has been studied in various other contexts (see  \cite{AF2004, DZ2003,FG1998,FGY1996,FGY1997,HPP2001,N} for some examples) and is related to other combinatorial parameters such as average distance \cite{FH1997}, packing chromatic number \cite{GHHHR2008}, injective chromatic number \cite{HKSS2002}, and strong chromatic index \cite{M2000}. It is known that determining $\alpha_k$ is NP-Hard in general \cite{K1993}.

In this article, we prove two spectral upper bounds for $\alpha_k$ that generalize two well-known bounds for the independence number: Cvetkovi\'{c}'s inertia bound \cite{C1972} and the Hoffman ratio bound (see \cite[Theorem 3.5.2]{BH2012} for example). Our motivation behind this work is obtaining non-trivial and tight generalizations of these classical results in spectral graph theory that depend on the parameters of the original graph and not of its higher powers. Note that $\alpha_k$ is the independence number of $G^k$, the $k$-th power of $G$. The graph $G^k$ has the same vertex set as $G$ and two distinct vertices are adjacent in $G^k$ if their distance in $G$ is $k$ or less. In general, even the simplest spectral or combinatorial parameters of $G^k$ cannot be deduced easily from the similar parameters of $G$ (see \cite{DG,DMS,Hegarty}). The bounds in our main results (Theorem \ref{theo:kindependencenumberbound} and Theorem \ref{hoffman theorem}) depend on the eigenvalues and eigenvectors of the adjacency matrix of $G$ and do not require the spectrum of $G^k$. We prove our main results in Section \ref{section:generalizedinetiabound} and Section \ref{section:hoffmanratiobound}. We construct infinite examples showing that the bound in Theorem \ref{theo:kindependencenumberbound} is tight, but we have not been able to determine whether or not equality can happen in Theorem \ref{hoffman theorem}. We conclude our paper with a detailed comparison of our bounds to previous work of Fiol \cite{F1999, F1997} and some directions for future work.

\section{Preliminaries}\label{preliminaries}
Throughout this paper $G=(V,E)$ will denote a graph (undirected, simple and loopless) on vertex set $V$ with $n$ vertices, edge set $E$ and adjacency matrix $A$ with eigenvalues $\lambda_{1}\geq \cdots \geq \lambda_{n}$. The following result was proved by Haemers in his Ph.D. Thesis (see \cite{H1995} for example).

\medskip
\begin{lema}[Eigenvalue Interlacing, \cite{H1995}]\label{theo:interlacing}
Let $A$ be a symmetric $n\times n$ matrix with eigenvalues
$\lambda_{1}\geq \lambda_{2}\geq \cdots \geq \lambda_{n}$. For some
integer $m< n$, let $S$ be a real $n\times m$ matrix such that
$S^{\top}S=I$ (its columns are orthonormal), and consider
the $m\times m$ matrix $B=S^{\top}AS$, with eigenvalues
$\mu_{1}\geq \mu_{2}\geq \cdots \geq \mu_{m}$. Then, the eigenvalues of $B$ interlace the eigenvalues
    of $A$, that is, $\lambda_{i}\geq \mu_{i}\geq \lambda_{n-m+i}$, for $1\leq i
\leq m$.

\end{lema}

\medskip
If we take $S=[\, I\ \ O\,]$, then $B$ is just a principal submatrix of $A$ and
we have:

\medskip
\begin{coro}\label{coro:2.2Haemerspaper}
If $B$ is a principal submatrix of a symmetric matrix $A$,
then the eigenvalues of $B$ interlace the eigenvalues of $A$.
\end{coro}

\section{Generalized inertia bound}\label{section:generalizedinetiabound}

Cvetkovi\'{c} \cite{C1972} (see also \cite[p.39]{BH2012} or \cite[p.205]{GR2001}) obtained the following upper bound for the independence number.
\begin{theo} [Cvetkovi\'{c}'s inertia bound]\label{theorem:Cvetkovicbound}
If $G$ is a graph, then
\begin{equation}\label{eq:Cvetkovicbound}
\alpha(G)\le \min\{|i:\lambda_i\ge 0|,|i:\lambda_i\le 0|\}.
\end{equation}
\end{theo}

Let $w_k(G)=\min_i (A^k)_{ii}$ be the minimum number of closed walks of length $k$ where the minimum is taken over all the vertices of $G$. Similarly, let $W_k(G)=\max_i (A^k)_{ii}$ be the maximum number of closed walks of length $k$ where the maximum is taken over all the vertices of $G$. Our first main theorem generalizes Cvetkovi\'{c}'s inertia bound which can be recovered when $k=1$.
\begin{theo}\label{theo:kindependencenumberbound}
Let $G$ be a graph on $n$ vertices. Then,
\begin{equation}\label{kindependencenumberbound}
\alpha_{k}(G)\le |\{i:\lambda_i^k\ge w_{k}(G)\}| \quad \text{and} \quad \alpha_{k}(G)\le |\{i:\lambda_i^k\le W_{k}(G)\}|.
\end{equation}
\end{theo}
\begin{proof}
Because $G$ has a $k$-independent set
$U$ of size $\alpha_k$, the matrix $A^{k}$
has a principal submatrix (with rows and columns corresponding to the vertices of $U$)
whose off-diagonal entries are $0$ and whose diagonal entries equal the number of closed walks of length $k$ starting at vertices of $U$. Corollary \ref{coro:2.2Haemerspaper} leads to the desired conclusion.
\end{proof}

We now describe why the above theorem can be considered a spectral result in nature. Our bounds are functions of the eigenvalues of $A$ and of certain counts of closed walks in $G$. We will now explain briefly how one may count closed walks in $G$ using only the eigenvalues and eigenvectors of $A$ (cf \cite{CDS} Section 1.8 or \cite{FGY1996}). Let $x_1,\cdots ,x_n$ be an orthonormal basis of eigenvectors for $A$ and let $e_1,\cdots , e_n$ be standard basis vectors. Then the number of closed walks of length $j$ from a vertex $v$ is given by
\[
e_v^T A^j e_v = \left( \sum_{i=1}^n \langle e_v, x_i\rangle x_i\right)^T A^j \left(\sum_{i=1}^n \langle e_v, x_i \rangle x_i \right) = \sum_{i=1}^n \langle e_v, x_i\rangle^2 \lambda_i^j.
\]
Therefore, one may apply our bounds with knowledge only of the eigenvalues and eigenvectors for $A$.  In fact, we note that one does not need to compute the eigenvectors of $A$ to count walks in $G$. It is enough to compute the idempotents which are given in terms of a polynomial in $A$ and the eigenvalues of $A$ \cite{fiol}.



\subsection{Construction attaining equality}\label{equality}

In this section, we describe a set of graphs for which Theorem \ref{theo:kindependencenumberbound} is tight. For $k,m \geq 1$ we will construct a graph $G$ with $\alpha_{2k+2}(G) = \alpha_{2k+3}(G) = m$.

Le $H$ be the graph obtained from the complete graph $K_n$ by removing one edge. The eigenvalues of $H$ are $\frac{n-3\pm \sqrt{(n+1)^2-8}}{2}, 0$ (each with multiplicity $1$), and $-1$ with multiplicity $n-3$. This implies $|\lambda_i(H)|<2$ for $i>1$.

Let $H_1,...,H_m$ be vertex disjoint copies of $H$ with $u_i, v_i\in V(H_i)$ and $u_i\not \sim v_i$ for $1\leq i\leq m$. Let $x$ be a new vertex. For each $1\leq i\leq m$, create a path of length $k$ with $x$ as one endpoint and $u_i$ as the other. Let $G$ be the resulting graph which has $nm + (k-2)m + 1$ vertices with $m\left( \binom{n}{2} -1\right) + mk$ edges.

Because the distance between any distinct $v_i$s is $2k+4$, we get that
\begin{equation}\label{extremal_lowerbound}
\alpha_{2k+2}(G) \geq \alpha_{2k+3}(G) \geq m.
\end{equation}

We will use Theorem \ref{theo:kindependencenumberbound} to show that equality occurs in \eqref{extremal_lowerbound} for $n$ large enough.

Starting from any vertex of $G$, one can find a closed walk of length $2k+2$ or $2k+3$ that contains an edge of some $H_i$. Therefore, $w_{2k+2}(G)\geq n-2$ and $w_{2k+3}(G)\geq n-2$.
Choose $n$ so that $n-2>\left(\sqrt{m} + 4\right)^{2k+3}$. If we can show that
\begin{equation}\label{eigenvalue bound for construction}
|\lambda_i(G)| \leq \sqrt{m}+4
\end{equation}
 for all $i>m$, then Theorem \ref{theo:kindependencenumberbound} will imply that $\alpha_{2k+3}(G) \leq \alpha_{2k+2}(G) \leq m$ and we are done. To show \eqref{eigenvalue bound for construction}, note that the edge-set of $G$ is the union of $m$ edge disjoint copies of $H$, the star $K_{1,m}$, and $m$ vertex disjoint copies of $P_{k-1}$. Since the star $K_{1,m}$ has spectral radius $\sqrt{m}$ and a disjoint union of paths has spectral radius less than $2$, applying the Courant-Weyl inequalities again, along with the triangle inequality, proves \eqref{eigenvalue bound for construction} and shows the tightness of our examples.

\section{Generalized Hoffman bound} \label{section:hoffmanratiobound}

The following bound on the independence number is an unpublished result of Hoffman known as the Hoffman's ratio bound (see \cite[p.39]{BH2012} or \cite[p.204]{GR2001}).

\begin{theo}[Hoffman ratio bound]
If $G$ is regular then \[\alpha (G)\leq n\frac{-\lambda_n}{\lambda_1-\lambda_n},\] and if a coclique $C$ meets this bound then every vertex not in $C$ is adjacent to precisely $-\lambda_n$ vertices of $C$.
\end{theo}

Let $G$ be a $d$-regular graph on $n$ vertices having an adjacency matrix $A$ with eigenvalues $d = \lambda_1 \geq \lambda_2 \geq \cdots \geq \lambda_n \geq -d$. Let $\lambda = \mathrm{max}\{|\lambda_2|, |\lambda_n|\}$. We use Alon's notation and say that $G$ is an $(n,d,\lambda)$-graph (see also \cite[p.19]{KS}). Let $\tilde{W}_k=\max_i \sum_{j=1}^{k} (A^j)_{ii}$ be the maximum over all vertices of the number of closed walks of length at most $k$. Our second theorem is an extension of the Hoffman bound to $k$-independent sets.

\begin{theo}\label{hoffman theorem}
Let $G$ be an $(n,d,\lambda)$-graph and $k$ a natural number. Then
\begin{equation}\label{kindependencenumberboundTAIT}
\alpha_k (G) \leq n\frac{\tilde{W}_k+\sum_{j=1}^k \lambda^j}{\sum_{j=1}^k d^j + \sum_{j=1}^k \lambda^j}.
\end{equation}
\end{theo}

The proof of Theorem \ref{hoffman theorem} will be given as a corollary of a type of Expander-Mixing Lemma (cf \cite{AC1988}). For $k$ a natural number, denote
\[
\lambda^{(k)} = \lambda + \lambda^2 + \cdots + \lambda^k,
\]
and
\[
d^{(k)} = d + d^2 + \cdots + d^k.
\]

\begin{theo}[$k$-Expander Mixing Lemma]\label{k expander mixing lemma}
Let $G$ be an $(n,d,\lambda)$-graph. For $S, T\subseteq G$ let $W_k(S,T)$ be the number of walks of length at most $k$ with one endpoint in $S$ and one endpoint in $T$. Then for any $S,T\subseteq V$, we have
\begin{equation*}
\left|W_k(S,T) - \frac{d^{(k)}|S||T|}{n}\right| \leq \lambda^{(k)}\sqrt{ |S||T| \left(1 - \frac{|S|}{n}\right)\left(1 - \frac{|T|}{n}\right)} < \lambda^{(k)} \sqrt{|S||T|}.
\end{equation*}
\end{theo}
\begin{proof}
Let $S,T\subset V(G)$ and let $\mathbf{1}_S$ and $\mathbf{1}_T$ be the characteristic vectors for $S$ and $T$ respectively.  Then
\[
W_k(S,T) = \mathbf{1}_S^t \left(\sum_{j=1}^k A^j\right) \mathbf{1}_T.
\]

Let $x_1,...,x_n$ be an orthonormal basis of eigenvectors for $A$. Then $\mathbf{1}_S = \sum_{i=1}^n \alpha_i x_i$ and $\mathbf{1}_T = \sum_{i=1}^n \beta_i x_i$, where $\alpha_i = \langle\mathbf{1}_S, x_i\rangle$ and $\beta_i = \langle\mathbf{1}_T, x_i\rangle$. Note that $\sum \alpha_i^2 = \langle\mathbf{1}_S, \mathbf{1}_S\rangle = |S|$ and similarly, $\sum \beta_i^2 = |T|$. Because $G$ is $d$-regular, we get that $x_1 = \frac{1}{\sqrt{n}} \mathbf{1}$ and so $\alpha_1 = \frac{|S|}{n}$ and $\beta_1 = \frac{|T|}{n}$.  Now, since $i\not=j$ implies $\langle x_i, x_j \rangle = 0$, we have

\begin{align*}
W_k (S,T) &= \left(\sum_{i=1}^n \alpha_i x_i \right)^t \left(\sum_{j=1}^k A^j\right) \left(\sum_{i=1}^n \beta_i x_i\right) \\
&= \sum_{i,j} (\alpha_i x_i)((\beta_j(\lambda_j + \lambda_j^2 + \cdots + \lambda_j^k) x_j) \\
&= \sum_{i=1}^n (\lambda_i + \lambda_i^2 + \cdots + \lambda_i^k)\alpha_i \beta_i \\
&=  \frac{d^{(k)}}{n}  |S||T| + \sum_{i=2}^n (\lambda_i + \lambda_i^2 + \cdots + \lambda_i^k)\alpha_i \beta_i
\end{align*}

Therefore, we have
\begin{align*}
\left| W_k(S,T) - \frac{d^{(k)}}{n}|S||T| \right| &= \left|\sum_{i=2}^n (\lambda_i + \lambda_i^2 + \cdots + \lambda_i^k)\alpha_i \beta_i\right|\\
&\leq \lambda^{(k)} \sum_{i=2}^n |\alpha_i \beta_i| \\
&\leq \lambda^{(k)} \left(\sum_{i=2}^n \alpha_i^2\right)^{1/2} \left(\sum_{i=2}^n \beta_i^2\right)^{1/2},
\end{align*}
where the last inequality is by Cauchy-Schwarz. Now since
\[
\sum_{i=2}^n \alpha_i^2 = |S| - \frac{|S|^2}{n^{2}}
\]
and
\[
\sum_{i=2}^n \beta_i^2 = |T| - \frac{|T|^2}{n^{2}},
\]
we have the result.
\end{proof}

Now we are ready to prove the bound of Theorem \ref{hoffman theorem}.

\begin{proof}[Proof of Theorem \ref{hoffman theorem}]
Let $S$ be a $k$-independent set in $G$ with $|S| = \alpha_k(G)$, and let $W_k(S,S)$ be equal to the number of closed walks of length at most $k$ starting in $S$. Theorem \ref{k expander mixing lemma} gives
\[
\frac{d^{(k)} |S|^2}{n} - W_k(S,S) \leq \lambda^{(k)} |S| \left(1 - \frac{|S|}{n}\right).
\]
Recalling that $\tilde{W}_k  = \mathrm{max}_i \sum_{j=1}^k (A^j)_{ii}$, we have $W_k(S,S) \leq |S| \tilde{W}_k$. This yields
\[
\frac{d^{(k)} |S|}{n} - \tilde{W}_k \leq \lambda^{(k)} \left(1- \frac{|S|}{n}\right).
\]
Solving for $|S|$ and substituting $|S| = \alpha_k$ gives
\[
\alpha_k \leq n\frac{\tilde{W}_k + \lambda^{(k)}}{d^{(k)} + \lambda^{(k)}}.
\]
\end{proof}

Note that the bound from Theorem \ref{hoffman theorem} behaves nicely if $\tilde{W}_k$ and $\lambda^{(k)}$ are small with respect of $d^{(k)}$. It is easy to see that $\tilde{W}_k\leq \frac{d^{k}-1}{d-1}$ (we expand $d$ in each step but in the last step we do not have any freedom since we assume that we are counting closed walks). Since $G$ is $d$-regular and we know that $\tilde{W}_k\leq d^{k-1}$, the above bound performs well for graphs with a good spectral gap.

\section{Concluding Remarks}\label{conclusion}

In this section, we note how our theorems compare with previous upper bounds on $\alpha_k$. Our generalized Hoffman bound for $\alpha_k$ is best compared with Firby and Haviland \cite{FH1997}, who proved that
if $G$ is a connected graph of order $n\geq 2$ then
\begin{equation}\label{Firby bound}
\alpha_k (G) \leq \frac{2(n-\epsilon)}{k+2-\epsilon}
\end{equation}
where $\epsilon \equiv k \pmod{2}$. If $d$ is large compared to $k$ and $\lambda = o(d)$, then Theorem \ref{hoffman theorem} is much better than \eqref{Firby bound} (this may be expected, as we have used much more information than is necessary for \eqref{Firby bound}). We note that almost all $d$-regular graphs have $\lambda = o(d)$ as $d\to\infty$.

In \cite{F1997}, Fiol (improving work from \cite{FG1998}) obtained the bound
\begin{equation}\label{Fiolbound}
\alpha_k (G) \leq \frac{2n}{P_k(\lambda_1)},
\end{equation}
when $G$ is a regular graph (later generalized to nonregular graphs in \cite{F1999}), and $P_k$ is the $k$-alternating polynomial of $G$. The polynomial $P_k$ is defined by the solution of a linear programming problem which depends on the spectrum of the graph $G$. It is nontrivial to compute $P_k$, and there is not a closed form for it making it difficult to compare to our theorems in general. However, the reader may check the Appendix to see that \eqref{Fiolbound}, Theorem \ref{hoffman theorem}, and Theorem \ref{theo:kindependencenumberbound} are pairwise incomparable. 

If $p$ is a polynomial of degree at most $k$, and $U$ is a $k$-independent set in $G$, then $p(A)$ has a principal submatrix defined by $U$ that is diagonal, with diagonal entries defined by a linear combination of various closed walk. Theorems \ref{theo:kindependencenumberbound} and \ref{hoffman theorem} are obtained by taking $p(A) = A^k$, but hold also for in general for other polynomials of degree at most $k$. It is not clear to us how to choose a polynomial $p$ to optimize the upper bound on $\alpha_k$ for general graphs as such polynomial will likely depend on the graph. Finally, we were able to construct graphs attaining equality in Theorem~\ref{theo:kindependencenumberbound} but not in Theorem~\ref{hoffman theorem}. We leave open whether the bound in Theorem~\ref{hoffman theorem} is attained for some graphs or can be improved in general.

\subsection*{Acknowledgments}
The authors would like to thank Randy Elzinga for helpful discussions. Some of this work was done when the first and third authors were at the SP Coding and Information School in Campinas, Brazil. We gratefully acknowledge support from UNICAMP and the school organizers. We would also like to thank an anonymous referee for helpful comments.

\section*{Appendix}
\appendix

The following tables compare Fiol's bound on $\alpha_k$, namely equation \eqref{Fiolbound}, with Theorem~\ref{theo:kindependencenumberbound} and Theorem~\ref{hoffman theorem}. We tested all of the named graphs in Sage that do not take any arguments, subject to the constraint that they are regular, connected, and have diameter greater than $k$. Table entries that say ``time" denote that the computation took longer than 60 seconds on a standard laptop. The parameter $\alpha_k$ is computationally hard to determine, and it is not clear how long it would take to calculate the table entries that timed out. A standard laptop was not able to compute $\alpha_2$ of the Balaban 11-cage in 1 hour. Part of the value of our theorems is that they give an efficient way to compute bounds for a parameter for which it may be infeasible to compute an exact value.

\subsection{$k=2$}
\begin{longtable}{lllll}
Graph name & Eq. \eqref{Fiolbound} & Thm. \ref{theo:kindependencenumberbound} & Thm. \ref{hoffman theorem} &  $\alpha_k$\\
\hline
Balaban 10-cage & $73$ & $43$ & $32$ & $17$ \\
Frucht graph & $9$ & $6$ & $6$ & $3$ \\
Meredith Graph & $115$ & $41$ & $20$ & $10$ \\
Moebius-Kantor Graph & $8$ & $10$ & $10$ & $4$ \\
Bidiakis cube & $7$ & $6$ & $5$ & $2$ \\
Gosset Graph & $2$ & $7$ & $8$ & $2$ \\
Balaban 11-cage & $123$ & $68$ & $41$ & time \\
Gray graph & $56$ & $33$ & $38$ & $11$ \\
Nauru Graph & $16$ & $15$ & $10$ & $6$ \\
Blanusa First Snark Graph & $16$ & $10$ & $8$ & $4$ \\
Pappus Graph & $9$ & $11$ & $14$ & $3$ \\
Blanusa Second Snark Graph & $16$ & $10$ & $8$ & $4$ \\
Brinkmann graph & $6$ & $10$ & $9$ & $3$ \\
Harborth Graph & $86$ & $30$ & $24$ & $10$ \\
Perkel Graph & $12$ & $19$ & $18$ & $5$ \\
Harries Graph & $73$ & $43$ & $32$ & $17$ \\
Bucky Ball & $86$ & $35$ & $23$ & $12$ \\
Harries-Wong graph & $73$ & $43$ & $32$ & $17$ \\
Robertson Graph & $4$ & $9$ & $6$ & $3$ \\
Heawood graph & $5$ & $8$ & $2$ & $2$ \\
Cell 600 & $92$ & $53$ & $18$ & $8$ \\
Cell 120 & $1018$ & $351$ & $302$ & time \\
Hoffman Graph & $6$ & $9$ & $10$ & $2$ \\
Sylvester Graph & $8$ & $14$ & $10$ & $6$ \\
Coxeter Graph & $16$ & $15$ & $13$ & $7$ \\
Holt graph & $10$ & $12$ & $14$ & $3$ \\
Szekeres Snark Graph & $70$ & $29$ & $25$ & $9$ \\
Desargues Graph & $13$ & $12$ & $10$ & $4$ \\
Horton Graph & $177$ & $60$ & $50$ & $24$ \\
Dejter Graph & $91$ & $64$ & $44$ & $16$ \\
Tietze Graph & $6$ & $6$ & $5$ & $3$ \\
Double star snark & $34$ & $17$ & $12$ & $6$ \\
Truncated Icosidodecahedron & $211$ & $75$ & $60$ & $26$\\
Durer graph & $9$ & $6$ & $5$ & $2$ \\
Klein 3-regular Graph & $54$ & $32$ & $22$ & $12$ \\
Truncated Tetrahedron & $6$ & $6$ & $5$ & $3$ \\
Dyck graph & $26$ & $20$ & $14$ & $8$ \\
Klein 7-regular Graph & $3$ & $6$ & $17$ & $3$ \\
Tutte 12-Cage & $132$ & $78$ & $44$ & time \\
Ellingham-Horton 54-graph & $92$ & $33$ & $32$ & $11$ \\
Tutte-Coxeter graph & $20$ & $18$ & $10$ & $6$ \\
Ellingham-Horton 78-graph & $148$ & $48$ & $38$ & $18$\\
Ljubljana graph & $132$ & $70$ & $44$ & $26$ \\
Tutte Graph & $76$ & $28$ & $21$ & $10$ \\
F26A Graph & $18$ & $16$ & $12$ & $6$ \\
Watkins Snark Graph & $74$ & $30$ & $25$ & $9$ \\
Flower Snark & $13$ & $11$ & $7$ & $5$ \\
Markstroem Graph & $29$ & $14$ & $11$ & $6$ \\
Wells graph & $6$ & $12$ & $22$ & $2$ \\
Folkman Graph & $11$ & $12$ & $10$ & $3$ \\
Foster Graph & $94$ & $56$ & $44$ & $21$ \\
McGee graph & $15$ & $13$ & $10$ & $5$ \\
Franklin graph & $6$ & $7$ & $6$ & $2$ \\
Hexahedron & $2$ & $5$ & $2$ & $2$ \\
Dodecahedron & $15$ & $10$ & $9$ & $4$ \\
Icosahedron & $2$ & $3$ & $7$ & $2$ \\
\end{longtable}

\subsection{$k=3$}
\begin{longtable}{lllll}
Graph name & Eq. \eqref{Fiolbound} & Thm. \ref{theo:kindependencenumberbound} & Thm. \ref{hoffman theorem} & $\alpha_k$\\
\hline
Balaban 10-cage & $41$ & $37$ & $36$ & $9$ \\
Frucht graph & $4$ & $5$ & $6$ & $2$ \\
Meredith Graph & $85$ & $35$ & $45$ & $7$ \\
Moebius-Kantor Graph & $2$ & $8$ & $8$ & $2$ \\
Balaban 11-cage & $72$ & $56$ & $59$ & $16$ \\
Gray graph & $27$ & $29$ & $35$ & $9$ \\
Nauru Graph & $6$ & $12$ & $14$ & $4$ \\
Blanusa First Snark Graph & $8$ & $8$ & $8$ & $2$ \\
Pappus Graph & $3$ & $9$ & $11$ & $3$ \\
Blanusa Second Snark Graph & $9$ & $7$ & $8$ & $2$ \\
Harborth Graph & $70$ & $27$ & $20$ & $6$ \\
Harries Graph & $41$ & $37$ & $35$ & $10$ \\
Bucky Ball & $62$ & $29$ & $30$ & $7$ \\
Harries-Wong graph & $41$ & $37$ & $35$ & $9$ \\
Cell 600 & $42$ & $45$ & $14$ & $3$ \\
Cell 120 & $847$ & $299$ & $287$ & time \\
Hoffman Graph & $2$ & $8$ & $11$ & $2$ \\
Coxeter Graph & $6$ & $11$ & $13$ & $4$ \\
Szekeres Snark Graph & $50$ & $24$ & $24$ & $6$ \\
Desargues Graph & $5$ & $10$ & $10$ & $2$ \\
Horton Graph & $162$ & $51$ & $50$ & $14$ \\
Dejter Graph & $42$ & $57$ & $56$ & $8$ \\
Double star snark & $19$ & $13$ & $15$ & $4$ \\
Truncated Icosidodecahedron & $180$ & $64$ & $60$ & $18$
\\
Durer graph & $4$ & $4$ & $7$ & $2$ \\
Klein 3-regular Graph & $28$ & $24$ & $29$ & $7$ \\
Dyck graph & $12$ & $17$ & $16$ & $4$ \\
Tutte 12-Cage & $72$ & $67$ & $77$ & $21$ \\
Ellingham-Horton 54-graph & $78$ & $29$ & $29$ & $8$ \\
Tutte-Coxeter graph & $7$ & $16$ & $20$ & $5$ \\
Ellingham-Horton 78-graph & $140$ & $42$ & $40$ & $11$
\\
Ljubljana graph & $80$ & $60$ & $63$ & $17$ \\
Tutte Graph & $62$ & $23$ & $21$ & $6$ \\
F26A Graph & $8$ & $14$ & $13$ & $3$ \\
Watkins Snark Graph & $54$ & $24$ & $24$ & $6$ \\
Flower Snark & $5$ & $9$ & $10$ & $2$ \\
Markstroem Graph & $19$ & $11$ & $10$ & $3$ \\
Wells graph & $2$ & $7$ & $13$ & $2$ \\
Folkman Graph & $4$ & $10$ & $15$ & $2$ \\
Foster Graph & $52$ & $48$ & $50$ & $15$ \\
McGee graph & $4$ & $10$ & $12$ & $2$ \\
Dodecahedron & $5$ & $7$ & $11$ & $2$ \\
\end{longtable}

\subsection{$k=4$}
\begin{longtable}{lllll}
Graph name & Eq. \eqref{Fiolbound} & Thm. \ref{theo:kindependencenumberbound} & Thm. \ref{hoffman theorem} &  $\alpha_k$\\
\hline
Balaban 10-cage & $21$ & $40$ & $32$ & $5$ \\
Meredith Graph & $63$ & $39$ & $20$ & $5$ \\
Balaban 11-cage & $37$ & $60$ & $39$ & $9$ \\
Gray graph & $17$ & $31$ & $14$ & $3$ \\
Harborth Graph & $55$ & $27$ & $13$ & $4$ \\
Harries Graph & $21$ & $40$ & $32$ & $5$ \\
Bucky Ball & $41$ & $30$ & $20$ & $6$ \\
Harries-Wong graph & $21$ & $40$ & $32$ & $5$ \\
Cell 600 & $16$ & $39$ & $14$ & $2$ \\
Cell 120 & $675$ & $309$ & $250$ & time \\
Szekeres Snark Graph & $33$ & $25$ & $17$ & $5$ \\
Desargues Graph & $2$ & $11$ & $10$ & $2$ \\
Horton Graph & $144$ & $55$ & $26$ & $8$ \\
Dejter Graph & $20$ & $59$ & $16$ & $2$ \\
Truncated Icosidodecahedron & $151$ & $70$ & $40$ & $11$
\\
Klein 3-regular Graph & $14$ & $25$ & $22$ & $4$ \\
Dyck graph & $5$ & $18$ & $14$ & $2$ \\
Tutte 12-Cage & $38$ & $72$ & $44$ & $9$ \\
Ellingham-Horton 54-graph & $63$ & $31$ & $12$ & $4$ \\
Ellingham-Horton 78-graph & $129$ & $44$ & $24$ & $6$ \\
Ljubljana graph & $46$ & $64$ & $44$ & $8$ \\
Tutte Graph & $48$ & $25$ & $20$ & $4$ \\
F26A Graph & $3$ & $14$ & $12$ & $2$ \\
Watkins Snark Graph & $37$ & $25$ & $19$ & $5$ \\
Markstroem Graph & $11$ & $11$ & $6$ & $3$ \\
Foster Graph & $25$ & $51$ & $44$ & $5$ \\
Dodecahedron & $2$ & $7$ & $9$ & $2$ \\
\end{longtable}

\end{document}